 \documentclass[final,1p,times]{elsarticle}

\usepackage{amssymb}
\usepackage{amsthm}
\usepackage{amsmath}

\journal{Applied Numerical Mathematics}

\newtheorem{proposition}{Proposition}
\newtheorem{corollary}{Corollary}

\newtheorem*{lemma*}{Lemma}
\newtheorem{theorem}{Theorem}

\newcommand{\argmin}{\mathop{\text{argmin}}}

\newcommand{\IM}{\mathop{\text{Im}}\nolimits}

\newcommand{\RR}{\mathbb R}
\newcommand{\CC}{\mathbb C}
\renewcommand{\leq}{\leqslant}
\renewcommand{\geq}{\geqslant}

\begin{document}

\begin{frontmatter}



\title{Stabilized explicit Adams-type methods}
\author{Vasily Repnikov}
\ead{repnikov@bsu.by}
\author{Boris Faleichik}
\ead{faleichik@bsu.by}
\author{Andrey Moysa}
\ead{moysa@bsu.by}

\address[bsu]{Department of Computational Mathematics, Belarusian State University, 4,~Nezavisimosti~avenue, 220030, Minsk, Belarus}

\author{}

\address{}

\begin{abstract}
In this work we present explicit Adams-type multistep methods with extended stability interval, which are analogous to the stabilized Chebyshev Runge--Kutta methods. It is proved that for any $k\geq 1$ there exists an explicit $k$-step Adams-type method of order one with stability interval of length $2k$. The first order methods have remarkably simple expressions for their coefficients and error constant. A damped modification of these  methods is derived. In general case to construct a $k$-step method of order $p$ it is necessary to solve a constrained optimization problem in which the objective function and $p$ constraints are second degree polynomials in $k$ variables. 
We calculate higher-order methods up to order six numerically and perform some numerical experiments to confirm the accuracy and stability of the methods.

\end{abstract}



\begin{keyword}
numerical ODE solution \sep
stiffness \sep
stability interval \sep
absolute stability \sep
multistep methods \sep
Adams-type methods \sep
explicit methods

\MSC 65L04 \sep 65L05 \sep 65L06



\end{keyword}

\end{frontmatter}


\section*{Introduction}
\label{s:intro}

A $k$-step  Adams method for the numerical integration of ODE system
\begin{equation}
  y'=f(t,y),\quad y(t_0)=y_0,
  \quad y:\RR\to\RR^n, \quad f:\RR\times \RR^n\to \RR^n,
\end{equation}
on uniform grid has the form
\begin{equation}
\label{eq:adams}
  y_{m+k}=y_{m+k-1}+\tau (\beta_0 f_m+\ldots+\beta_{k-1} f_{m+k-1}),
\end{equation}
where $m$ is the step number, ${y_j\approx y(t_0+j\tau)}$, ${f_j=f(t_0+j\tau , y_j)}$, $\tau$ is a discretization step, $\{\beta_j\}$ are the method coefficients.

A conventional means of analyzing the linear stability of a multistep method is construction of stability region $S\subset \CC$ such that for all $\lambda\tau \in S$ the numerical solution of the model linear problem
\[
 y'=\lambda y,\quad \lambda \in\CC,
\]
remains bounded for all $m$. The equivalent requirement is that all roots $\zeta_j$ of the characteristic equation   
\begin{equation}\label{eq:chareq}
	\rho(\zeta)-\lambda\tau \sigma(\zeta)=0
\end{equation}
lie within the unit disc on the complex plane, and all the roots of modulus one are simple \cite{HW}. Here $\rho$ and $\sigma$ are the standard generating polynomials which in our case have the form 
\begin{equation}
	\rho(\zeta)=\zeta^k-\zeta^{k-1},\quad
	\sigma(\zeta)=\sum_{j=0}^{k-1}\beta_j \zeta^j.
\end{equation}
The stability interval of a method is the largest interval of the form $[-\ell,0]$ contained in $S$. Here $\ell \geq 0$ is the value which will be referred to as the length of stability interval.
As is known, stability intervals of the classical explicit Adams methods are small and get smaller with the growth of $k$, so these methods are not suitable for stiff problems. 
The purpose of the present research is to construct explicit multistep methods of Adams type \eqref{eq:adams} of order $p<k$ with increased lengths of stability interval. Putting it in other words, we develop a multistep analogs of the well-known Chebyshev Runge--Kutta methods \cite{lebedev}, \cite{sommeijer}, \cite{Abdulle}. 

The main obstacle on the way of construction of  multistep methods for stiff problems is the dependence of error constant on the size of stability region, which was investigated by Jeltsch and Nevanlinna, \cite{Jeltsch1981}, \cite{Jeltsch1982}. On the one hand, due to \cite[Theorem 4.2]{Jeltsch1981}, for any $k>1$, $\alpha<\pi/2$ and $R>0$ there exists an explicit linear multistep method of order $k-1$ such that the method is stable in the set
\begin{equation*}
  	\{\mu\in\CC\::\: |\mu|\leq R, \: |\arg(-\mu)|\leq \alpha\}.
\end{equation*}  
Unfortunately, methods which stability regions contain large disks of the form $\{\mu\in\CC \::\: |\mu+R|\leq R\}$ are useless in practice due to huge error constants (norms of Peano kernels) \cite[Theorem 4.1]{Jeltsch1982}, \cite[V.2, Theorem 2.6]{HW}. On the other hand, in case of long stability intervals the lower bounds for error constant are less restrictive. Namely, \cite[Theorem 4.4]{Jeltsch1982} gives a lower bound for Peano kernel norms for explicit multistep methods in the form of $C_k \ell$, where ${C_k>0}$. Another interesting fact from \cite[Theorem 4.7]{Jeltsch1982} is that for explicit $k$-step methods of order $k-1$ with $\ell>2$ the error constant has lower bound equal to $\delta_{k-1}(\ell-2)/2^{k-1}$, where $(\delta_k)$ is a decreasing sequence with $\delta_1=0.5$ . Thus we hope that explicit multistep methods with extended stability interval can have reasonable error constants (see Table \ref{tab:errconst}).

The material is organized as follows. In Section \ref{s:strategy} we describe the general framework of the methods construction, Sections \ref{s:order1} and \ref{s:order1-damp} are devoted to the first order methods and their damped modifications. Higher order methods are discussed in Section \ref{s:high-order}, Section \ref{s:num-exp} contains the results of numerical experiments. In the last section we discuss the obtained results and make final conclusions. 

\section{Optimization strategy}\label{s:strategy}

The conventional way of stability region construction is to find a root locus curve $\mathcal C$ defined as
\begin{equation}\label{eq:locus}
	\mathcal C=\{\mu_\beta(e^{i\varphi})\::\: \varphi\in[0,2\pi)\},
\end{equation}
where the function $\mu_\beta:\CC\to\CC$ maps a root of the characteristic equation \eqref{eq:chareq} to the corresponding value of $\lambda\tau$:
\begin{align}\label{eq:mu}
	\mu_\beta(\zeta)
	=\frac{\rho(\zeta)}{\sigma(\zeta)}.
\end{align}
The subscript $\beta$ indicates the dependence on the coefficients of method \eqref{eq:adams} which we are to determine.
From the definition of stability region it follows that
\(
	\partial S \subseteq \mathcal C
\)
and 
\begin{equation}
	\ell\leq -\mu_\beta(-1),
\end{equation}
thus the optimization problem can be stated as
\begin{equation}\label{eq:argmin}
	\beta^*=\argmin_{\beta\in\mathcal F\cap\mathcal P} \mu_\beta(-1),
\end{equation}
where $\mathcal P$ is a set of coefficients which satisfy the posed order conditions, and $\mathcal F$ is a feasible set of coefficients with a desired shape of the root locus curve. This set is defined as follows.

Primarily we'd like to have $\ell=-\mu_\beta(-1)$ for all ${\beta\in\mathcal F}$. To assure this we require the locus curve \eqref{eq:locus} not to cross the real axis before the parameter $\varphi$ reaches~$\pi$. This condition triggers the following definition for the feasible set:
\begin{equation}\label{eq:feasible}
	\mathcal F=\bigl\{\beta\in\RR^k \::\:
	\IM \mu_\beta(e^{i\varphi})\geq 0 \quad \forall \varphi\in(0,\pi)
	\bigr\}.
\end{equation}

The main question now is how to find a parametization of $\mathcal F$ which allows for reducing \eqref{eq:argmin}, \eqref{eq:feasible} to some handleable form. We start from noting that 
\begin{equation}
	\IM \mu_\beta(e^{i\varphi})\geq 0 
	\quad \Leftrightarrow\quad
	\nu_\beta(\varphi)\geq 0,
\end{equation}
where
\begin{equation}\label{eq:nu}
	\nu_\beta(\varphi)=\IM \rho(e^{i\varphi})\overline{\sigma(e^{i\varphi})}=\sum_{j=1}^k (\beta_{k-j}-\beta_{k-j-1})\sin j \varphi.
\end{equation}
Here and further we set $\beta_j=0$ for all $j<0$ and $j>k-1$. By utilizing the Chebyshev polynomials of the second kind $U_j$,  
\begin{equation}
	U_{j-1}(\cos \varphi)\sin\varphi=\sin j \varphi,
\end{equation}
and using the power reduction formulae for the powers of $\cos \varphi$, \eqref{eq:nu} can be represented as
\begin{equation}\label{eq:nu2}
 	\nu_\beta(\varphi)=\sin\varphi \sum_{j=0}^{k-1}a_j \cos j \varphi
 \end{equation} 
 with some ${a_j\in\RR}$.
Since we need \(\nu_\beta(\varphi)\) to be nonnegative on $(0,\pi)$, the following result from \cite[Lemma 6.1.3]{Daubechies} comes in handy.
\begin{lemma*}
	For any nonnegative trigonometric polynomial $A$ of the form
	\begin{equation*}
		A(\varphi)=\sum_{j=0}^k a_j \cos j\varphi, \quad a_j\in \RR,
	\end{equation*}
	there exists a trigonometric polynomial 
	\begin{equation*}
		B(\varphi)=\sum_{j=0}^k b_j e^{i j \varphi}, \quad b_j \in \RR,
	\end{equation*}
	such that $A(\varphi)=|B(\varphi)|^2$.
\end{lemma*}

From this lemma it follows that all feasible trigonometric polynomials $\nu_\beta$ have the form
\begin{equation}\label{eq:nu3}
	\nu_\beta(\varphi)=\sin\varphi\left|\sum_{j=0}^{k-1}b_j e^{i j \varphi}\right|^2=\sin\varphi\sum_{j,l=0}^{k-1}b_jb_l \cos(j-l)\varphi
\end{equation}
with $b_j\in\RR$. To complete the transformation of the optimization problem we must express the original coefficients $\{\beta_j\}$ in terms of $\{b_j\}$. This can be done by converting \eqref{eq:nu2} to the same basis of $\sin j\varphi$ as \eqref{eq:nu}:
\begin{equation*}
	\sin\varphi \sum_{j=0}^{k-1}a_j \cos j \varphi=\frac12\sum_{j=0}^{k-1}a_j (\sin(j+1)\varphi - \sin(j-1)\varphi).
\end{equation*}
By equating this expression with \eqref{eq:nu} it is straightforward to get

\begin{subequations}\label{eq:b-beta}
\begin{align}
	&\beta_j=\frac12(\tilde a_{j-1}+\tilde a_j),\quad j=0,1,\ldots, k-2,\\
	&\beta_{k-1}=\tilde a_{k-1}+\frac{\tilde a_{k-2}}{2}.
\end{align}
Here for clarity $\tilde a_j=a_{k-1-j}$, $\tilde a_{-1}=0$. On the other hand, from \eqref{eq:nu2}, \eqref{eq:nu3} we have

\begin{align}
	&\tilde a_j=2\sum_{l=0}^j b_l b_{k-1+l-j},\quad j=0,1,\ldots,k-2,\\
	&\tilde a_{k-1}=\sum_{j=0}^{k-1}b_j^2.
\end{align}
\end{subequations}
Transformations \eqref{eq:b-beta} define the required mapping \[T:b\mapsto \beta,\] where $b=(b_0,\ldots,b_{k-1})$, $\beta=(\beta_0,\ldots,\beta_{k-1})$.

Now let us derive the form of objective function \eqref{eq:argmin},
\begin{equation}\label{eq:mu-1}
 	\mu_\beta(-1)=\frac{2 (-1)^k}{\sum_{j=0}^{k-1} (-1)^j\beta_j},
 \end{equation} 
in terms of $b$. 
By direct application of \eqref{eq:b-beta} we have
\begin{equation}
 	\sum_{j=0}^{k-1} (-1)^j\beta_j=(-1)^k \sum_{j=0}^{k-1}b_j^2,
\end{equation} 
so the initial optimization problem \eqref{eq:argmin}, \eqref{eq:feasible} finally takes the surprisingly simple form
\begin{subequations}\label{eq:opt-b}
\begin{equation}\label{eq:b-argmin}
	b^*=\argmin_{b\in \mathcal P'_p} \sum_{j=0}^{k-1}b_j^2,
\end{equation}
where $\mathcal P'_p$ is the order $p$ restriction set:
\begin{align}
	\label{eq:b-set}
	&\mathcal P'_p=\{b\in\RR^k \::\: G_q(T(b))=0,\quad q=1,2,\ldots,p\},\\
	\label{eq:b-order1}
	&G_1(\beta)=\sum_{j=0}^{k-1}\beta_j-1,\\
	\label{eq:b-orderp}
	&G_q(\beta)=\sum_{j=0}^{k-1}(1-k+j)^{q-1}\beta_j-\frac{1}{q},\quad q>1.
\end{align}
\end{subequations}

\section{First order methods}\label{s:order1}
\begin{theorem}
	For any number of steps $k\geq 1$ there exists a first-order explicit Adams-type method with stability interval of length $2k$. The method has the form
	\begin{equation}\label{eq:order1-opt}
		y_{m+k}=y_{k+m-1}+ \frac{\tau}{k^2}\Bigl(f_{k+1}+3 f_{k+1}+\ldots+(2k-1) f_{k-1}\Bigr),
	\end{equation}
	i.~e. $\beta_j=(2j+1)/k^2$. 
\end{theorem}
\begin{proof}
	Directly applying \eqref{eq:b-beta} to the first order condition \eqref{eq:b-order1} we have
	\begin{equation*}
		\beta_0+\beta_1+\ldots+\beta_{k-1}=a_0+a_1+\ldots+a_{k-1}=1.
	\end{equation*}
	On the other hand, by construction form \eqref{eq:nu2}, \eqref{eq:nu3} we have 
	\begin{equation}
		\sum_{j=0}^{k-1}a_j = \sum_{j,l}^{k-1} b_j b_l=\left(\sum_{j=1}^{k-1}b_j\right)^2.
	\end{equation}
	Thus the optimization problem \eqref{eq:opt-b} in the case of ${p=1}$ takes the form
	\begin{subequations}\label{eq:first-order-problem}
	\begin{align}
		& b^*=\argmin_{b\in \RR^k} (b_0^2+b_1^2+\ldots+b_{k-1}^2),\\
		& \text{subject to}\quad (b_0+b_1+\ldots+b_{k-1})^2=1.
	\end{align}
	\end{subequations}
	Solving this problem by the method of Lagrange multipliers we get
	$b^*_j=k^{-1}$ for all $j$ and then by \eqref{eq:b-beta} obtain 
	\begin{equation*}
		\beta^*=T(b^*)=\left(\frac1{k^2},\frac3{k^2},\ldots,\frac{2k-1}{k^2}\right).
	\end{equation*}
	By construction of the method from \eqref{eq:mu-1} and since
	\begin{equation}\label{eq:sum-1beta}
		\sum_{j=0}^{k-1}(-1)^j \frac{2j+1}{k^2}=(-1)^{k-1}\frac1k,
	\end{equation}
	 we have $\ell=-\mu_{\beta^*}(-1)=2k$.
\end{proof}
There is an interesting parity of the above result with the case of $s$-stage Chebyshev Runge--Kutta methods of order 1, which require $s$ evaluations of $f$ per step and have stability interval equal to $2s^2$ \cite{HW}, \cite{lebedev}. This allows us to suppose that the achieved length of $2k$ is the largest possible for explicit first-order multistep methods. 

\paragraph{Error constant} According to \cite{HNW} we define the error constant of the multistep method as
\begin{equation}
	C=C_{p+1}/\sigma(1), \quad\text{where}\quad
	C_{p+1}=\frac{1}{(p+1)!}\sum_{j=0}^k \left(\alpha_j j^{p+1} - (p+1)\beta_j j^p\right).
\end{equation}
It is easy to calculate this constant in our case.
\begin{proposition}
	The error constant of the optimized first-order methods is equal to
	\begin{equation}
		C=\frac k3 + \frac{1}{6k}.
	\end{equation}
	\begin{proof}
		Since $\beta_j=(2j+1)/k^2$, $\alpha_k=1$, $\alpha_{k-1}=-1$, $\sigma(1)=1$, we have
		\begin{equation}
			C=\frac12\left(k^2-(k-1)^2-\frac{2}{k^2}\sum_{j=1}^{k-1}j(2j+1)\right)=\frac k3 + \frac{1}{6k}.
		\end{equation}
	\end{proof}
\end{proposition}

\section{First order methods with damping}\label{s:order1-damp}

Analogously to the Chebyshev RK methods, in order to pull the root locus curve away from the real axis for $\varphi\in(0,\pi)$ it is necessary to perform a damping transformation with the constructed methods.

Using \eqref{eq:nu}, \eqref{eq:nu2} consider  \eqref{eq:mu} and represent 
\begin{equation*}
	\IM \mu_\beta(e^{i\varphi}) = Q(\varphi)\sin \varphi ,
\end{equation*}
where 
\begin{equation}\label{eq:Q}
	Q(\varphi)=\frac{\sum_{j=0}^{k-1}a_j \cos j \varphi}{|\sigma(e^{i\varphi})|^2}=\frac{\sum_{j=0}^{k-1}a_j \cos j \varphi}{\sum_{j=0}^{k-1}\delta_j \cos j \varphi},
\end{equation}
and
\begin{equation}\label{eq:deltas}
	\delta_0=\sum_{j=0}^{k-1}\beta_j^2,\quad \delta_j=2\sum_{l=0}^j \beta_l \beta_{j+l}.
\end{equation}
Recall that the connection between coefficients $a_j$ and $\beta_j$ is described by the first two equalities of \eqref{eq:b-beta}.

Let $\hat Q(\varphi)$ be the damped method's counterpart of \eqref{eq:Q}. We define it as
\begin{equation}
	\hat Q(\varphi)=\frac{\sum_{j=0}^{k-1}\hat a_j \cos j \varphi}{\sum_{j=0}^{k-1}\hat \delta_j \cos j \varphi}=C(Q(\varphi)+\varepsilon),
\end{equation}
where $\varepsilon$ controls the `shift' from the real axis and $C$ is a scaling constant to be determined. Then we have $\hat a_j = C a'_j$,
\begin{equation}
	a_j'=a_j+\varepsilon \delta_j.
\end{equation}
Now we use this equality together with \eqref{eq:b-beta} and get
\begin{align*}
 	&\beta_j'=\beta_j+\frac{\varepsilon}{2}(\delta_{k-j}+\delta_{k-j-1}),\quad j=0,\ldots,{k-2},\\
 	&\beta_{k-1}'=\beta_{k-1}+\frac{\varepsilon}{2}\delta_1+\varepsilon\delta_0,\quad \delta_k=0.
 \end{align*} 
The coefficients $\hat \beta_j$ of the sought damped method are expressed as $\hat\beta_j=C\beta'_j$. To keep the order of the method equal to one, the constant $C$ should be equal to $\left(\sum_{j=0}^{k-1}\beta_j'\right)^{-1}$. By \eqref{eq:deltas} we have
\begin{equation}
	\sum_{j=0}^{k-1}\beta_j'=\sum_{j=0}^{k-1}\beta_j+\varepsilon\sum_{j=0}^{k-1}\delta_j=\sum_{j=0}^{k-1}\beta_j+\varepsilon\sum_{j,l=0}^{k-1}\beta_j\beta_l.
\end{equation}
Since ${\sum_{j=0}^{k-1}\beta_j=1}$ we finally obtain the following formulae for the coefficients of the damped method:
\begin{subequations}\label{eq:damped-beta}
\begin{align}
 	&\hat\beta_j=\frac{\beta_j + \varepsilon\Delta_j}{1+\varepsilon},\quad j=0,\ldots,{k-1};\label{eq:damped-beta-a}\\
 	&\text{where}\quad \Delta_j=\frac12(\delta_{k-j}+\delta_{k-j-1}),\quad j=0,\ldots,{k-2},\\
 	&\phantom{\text{where}\quad }\Delta_{k-1}=\frac12\delta_{1}+\delta_0.
\end{align}
\end{subequations}
The values of $\Delta_j$ for $k$ from 2 to 10 for the optimized first-order method \eqref{eq:order1-opt} are presented in Table \ref{table:Deltas}. The stability region boundaries of the one-step methods together with their damped counterparts are displayed at Figure \ref{fig:order1}. 
\begin{proposition}
	Stability interval of the damped method \eqref{eq:damped-beta} with $\beta_j=\beta^*_j=\frac{2j+1}{k^2}$ is equal to $[-\ell_\varepsilon, 0]$, where
	\begin{equation}
	 	\ell_\varepsilon=\frac{6(1+\varepsilon)k^3}{\varepsilon(4k^2-1)+3k^2}.
	\end{equation} 
\end{proposition}
\begin{proof}
	Let us compute
	\begin{equation*}
		\sum_{j=0}^{k-1}(-1)^j \hat \beta_j=(1+\varepsilon)^{-1}
		\left(\sum_{j=0}^{k-1}(-1)^j \beta^*_j+\varepsilon(-1)^{k-1}\delta^*_0\right).
	\end{equation*}
	The first term has already been calculated in \eqref{eq:sum-1beta} and the second is determined as
	\begin{equation*}
		\delta^*_0=\sum_{j=0}^{k-1} (\beta^*_j)^2=\sum_{j=0}^{k-1} \frac{(2j+1)^2}{k^4}=\frac{4k^2-1}{3k^3}.
	\end{equation*}
	From here we finally get
	\begin{equation*}
		\mu_{\hat\beta}(-1)=\frac{2(-1)^k}{\sum_{j=0}^{k-1} (-1)^j\hat\beta_j}=-\frac{6(1+\varepsilon)k^3}{\varepsilon(4k^2-1)+3k^2}.
	\end{equation*}
\end{proof}

\begin{corollary}
Asymptotic length of the damped one-step method is 
\begin{equation}
	\lim_{\varepsilon\to\infty} \ell_\varepsilon=\frac{3}{2}k.
\end{equation}
\end{corollary}
\begin{table*}[t]
\[
\renewcommand{\arraystretch}{1.5}
\begin{array}{c|ccccccccccc}
k & \Delta_0 & \Delta_1 & \Delta_2 &\Delta_3  &\Delta_4  &\Delta_5  &\Delta_6  &\Delta_7  &\Delta_8  &\Delta_{9}  \\
\hline
2& \frac{3}{16} & \frac{13}{16} & \text{} & \text{} & \text{} & \text{} & \text{} &
   \text{} & \text{} & \text{} \\
3& \frac{5}{81} & \frac{23}{81} & \frac{53}{81} & \text{} & \text{} & \text{} & \text{} &
   \text{} & \text{} & \text{} \\
 4& \frac{7}{256} & \frac{33}{256} & \frac{79}{256} & \frac{137}{256} & \text{} & \text{}
   & \text{} & \text{} & \text{} & \text{} \\
5& \frac{9}{625} & \frac{43}{625} & \frac{21}{125} & \frac{187}{625} & \frac{281}{625} &
   \text{} & \text{} & \text{} & \text{} & \text{} \\
 6&\frac{11}{1296} & \frac{53}{1296} & \frac{131}{1296} & \frac{79}{432} &
   \frac{121}{432} & \frac{167}{432} & \text{} & \text{} & \text{} & \text{} \\
 7& \frac{13}{2401} & \frac{9}{343} & \frac{157}{2401} & \frac{41}{343} & \frac{445}{2401}
   & \frac{89}{343} & \frac{813}{2401} & \text{} & \text{} & \text{} \\
8& \frac{15}{4096} & \frac{73}{4096} & \frac{183}{4096} & \frac{337}{4096} &
   \frac{527}{4096} & \frac{745}{4096} & \frac{983}{4096} & \frac{1233}{4096} & \text{}
   & \text{} \\
 9& \frac{17}{6561} & \frac{83}{6561} & \frac{209}{6561} & \frac{43}{729} &
   \frac{203}{2187} & \frac{289}{2187} & \frac{1153}{6561} & \frac{1459}{6561} &
  \frac{1777}{6561} & \text{} \\
 10& \frac{19}{10000} & \frac{93}{10000} & \frac{47}{2000} & \frac{437}{10000} &
   \frac{691}{10000} & \frac{989}{10000} & \frac{1323}{10000} & \frac{337}{2000} &
   \frac{2067}{10000} & \frac{2461}{10000} \\
\end{array}
\]
\caption{The values of $\Delta_j$ for determining the coefficients of the optimized first-order methods with damping \eqref{eq:damped-beta-a}, $\beta_j=(2j+1)/k^2$.}\label{table:Deltas}
\end{table*}

\section{Higher order methods}\label{s:high-order}

To construct a stabilized $k$-step Adams-type method of order $p$ one should use the general form of optimization problem \eqref{eq:opt-b} with mapping $T$ specified by \eqref{eq:b-beta}. For example, for $k=5$, $p=4$, the problem in terms of $b_j$ takes the form
\begin{equation*}
\begin{array}{l}
\text{minimize}\quad b_0^2+b_1^2+b_2^2+b_3^2+b_4^2\quad \text{subject to}\\
 \left(b_0+b_1+b_2+b_3+b_4\right){}^2=1, \\
 b_2 b_3+\left(3 b_2+b_3\right) b_4+b_1 \left(b_2+3 b_3+5 b_4\right)+b_0 \left(b_1+3 b_2+5 b_3+7 b_4\right)=-\frac{1}{2}, \\
 b_2 b_3+\left(5 b_2+b_3\right) b_4+b_1 \left(b_2+5 b_3+13 b_4\right)+b_0 \left(b_1+5 b_2+13 b_3+25 b_4\right)=\frac13,
   \\
 b_2 b_3+\left(9 b_2+b_3\right) b_4+b_1 \left(b_2+9 b_3+35 b_4\right)+b_0 \left(b_1+9 b_2+35 b_3+91 b_4\right)=-\frac{1}{4}.\\
\end{array} 	
 \end{equation*} 
 The symbolic solution of this problem yielded by Wolfram 
 \emph{Mathematica} after transforming back to the initial variables $\beta_j$ is
 \begin{equation*}
 \beta^*=\left(-\frac{1}{4},\frac{5}{8},\frac{1}{24},-\frac{35}{24},\frac{49}{24}\right)
 \end{equation*}
with $\ell=0.75$, to compare with $\ell=0.3$ of the classical explicit Adams method of order 4. Another neat example is the 5-step method of order 2:
\begin{equation}
	\beta^*=\left(-\frac{1}{8} \left(3-\sqrt{5}\right),-\frac{3}{4} \left(\sqrt{5}-2\right),0,\frac{7}{4}
   \left(\sqrt{5}-2\right),\frac{9}{8} \left(3-\sqrt{5}\right)\right)
\end{equation}
with $\ell=2+\frac{4}{\sqrt 5}\approx 3.789$. The stability regions of these and the rest of the 5-step methods are shown at the uppermost part of figure \ref{fig:k567}. 

Unfortunately 
it is not always possible to obtain solution symbolically, thus we compute the coefficients of our methods numerically using \emph{Mathematica}'s function \texttt{NMinimize}, see the corresponding code in \ref{app:code}. We used 50-digit working precision and computed the $(k,p)$ methods for $k$ from 3 to 10 and $p$ from 2 to $k$ (with the latter value corresponding to the classical Adams methods). The results with 20-digit accuracy are displayed at Tables \ref{tab:o2-5} and \ref{tab:o6-9}. It is interesting that the (4,3) method coincides with \cite[(5.4)]{XuZhao}. The stability regions of 5-, 6- and 7-step methods are shown at Figure \ref{fig:k567}. 

To assess the accuracy of the obtained coefficients we checked the magnitude of the order residuals $G_q(\beta^*)$, see \eqref{eq:b-order1}, \eqref{eq:b-orderp}. In all convergent cases these residuals do not exceed $10^{-19}$. Note that \texttt{NMinimize} failed to converge in the following cases: $(k,p)=(7,6)$ and for all ${p\geq 7}$. Our hypothesis is that for these $(k,p)$ combinations the Adams-type methods satisfying our restrictions do not exist. Note that (11, 7) method seem to exist and has microscopic ${\ell\approx0.051}$ which is just slightly more than ${\ell\approx0.0465}$ of the $(7,7)$ case. The error constants of all the calculated methods are presented at Table \ref{tab:errconst}.

\begin{table}
\small
\begin{center}
	\begin{tabular}{c|llllll}
	$k$  & $p=1$ & $p=2$ & $p=3$ & $p=4$ & $p=5$ & $p=6$\\
	\hline
	2 & 0.75 &  &  &  &  &    \\
	3 & 1.0556 & 0.66667 &  &  &  &    \\
	4 & 1.375 & 1.0380 & 0.62500 &  &  &   \\
	5 & 1.7 & 1.5208 & 1.0227 & 0.59861 &  &    \\
	6 & 2.0278 & 2.1128 & 1.5972 & 1.0120 & 0.57928 &    \\
	7 & 2.3571 & 2.8134 & 2.3814 & 1.6471 & 1.0032 &  \\
	8 & 2.6875 & 3.6223 & 3.4092 & 2.5751 & 1.6825 & 0.99505 \\
	9 & 3.0185 & 4.5392 & 4.7148 & 3.8788 & 2.7235 & 1.7079   \\
	10& 3.35 & 5.5643 & 6.3328 & 5.6524 & 4.2616 & 2.8403 \\
	\end{tabular}
\end{center}
\caption{Error constants of the stabilized Adams-type methods.}
\label{tab:errconst}
\end{table}

\section{Numerical experiments}\label{s:num-exp}

The purpose of the experiment is to verify accuracy and stability properties of the constructed stabilized Adams methods. 
We also display results of classic implicit Adams methods of corresponding orders, which have longer stability intervals than their classical explicit counterparts.
In all our experiments we use constant step size and reference solutions computed by Wolfram \emph{Mathematica}'s \texttt{NDSolve}. The starting points were taken from this reference solution. For each method we perform a series of constant-step integrations with decreasing step size $\tau$ and calculate maximum norm of the error at the endpoint.  Missing points on the convergence diagrams mean that the error is too large due to instability of the method for the particular value of $\tau$. 
\begin{figure}
	\begin{center}
		\includegraphics[width=6.7cm]{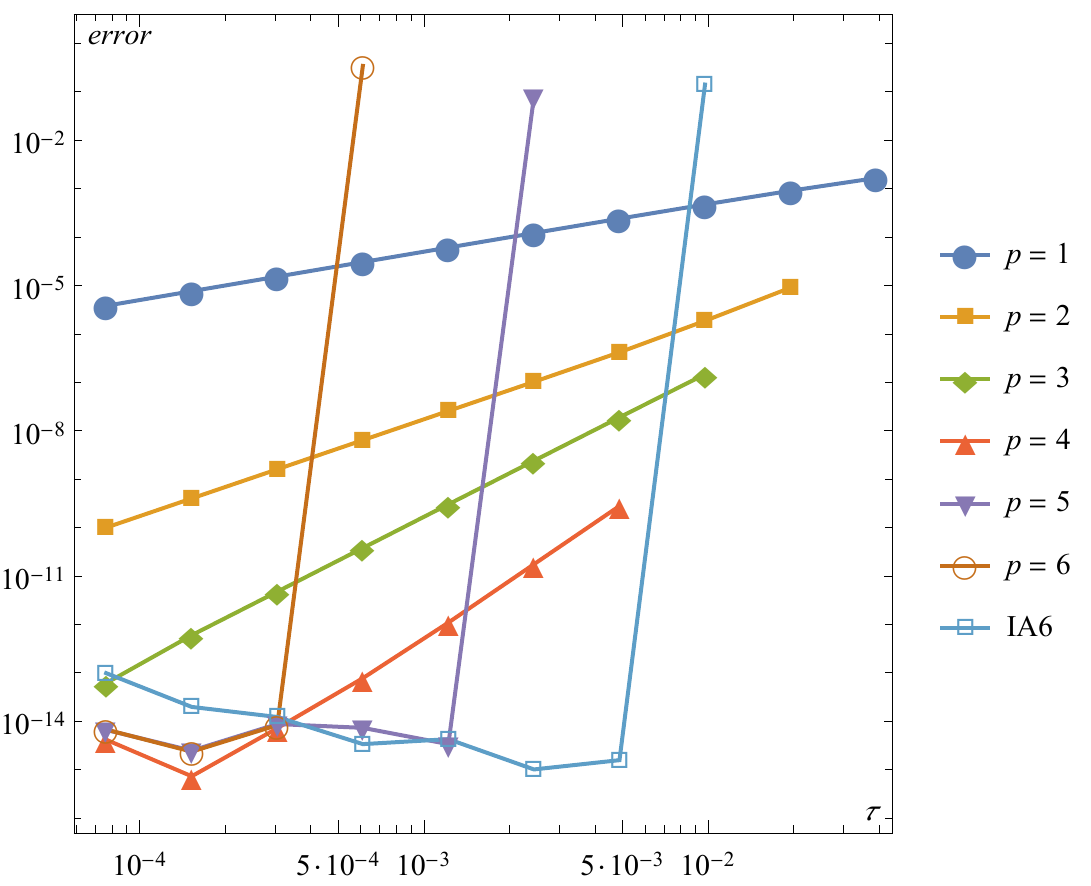}
		\includegraphics[width=6.7cm]{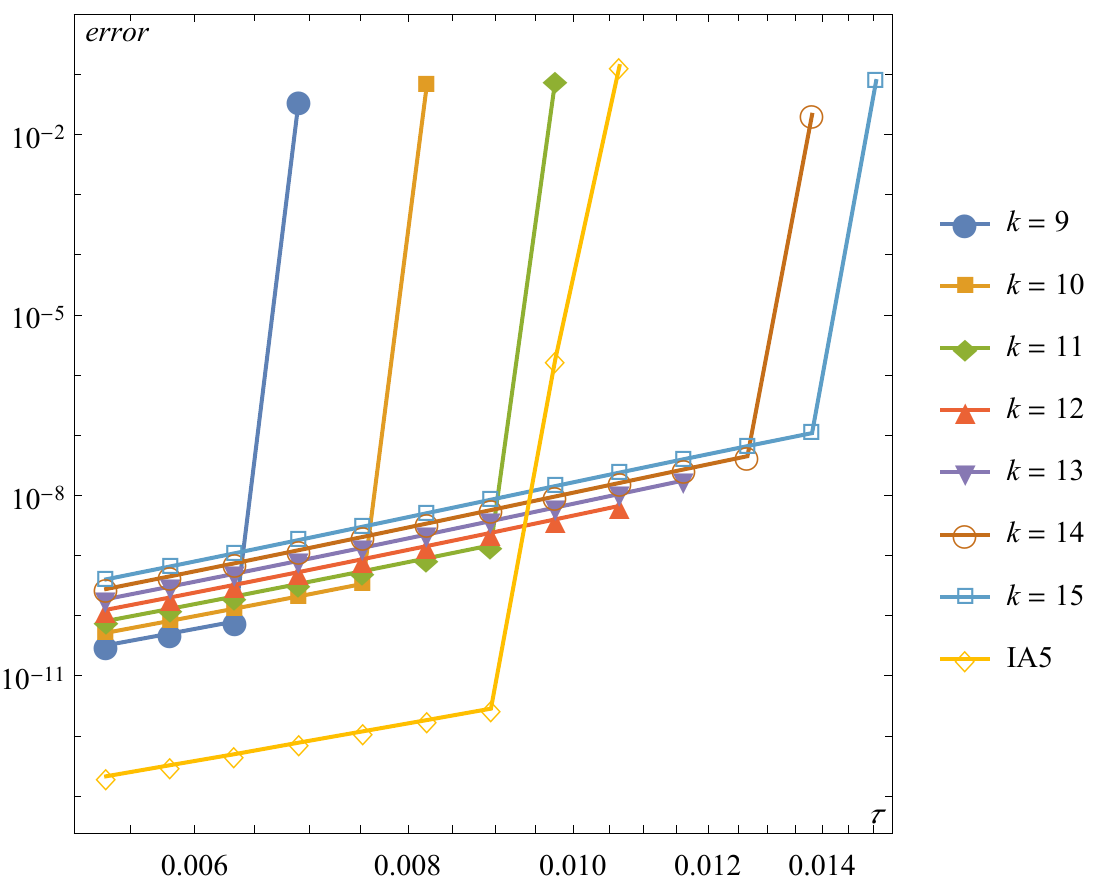}
	\end{center}
	\caption{Numerical experiment with HIRES problem.  Six-step stabilized methods and implicit Adams method of order~6 (left). Stabilized methods of order 5 and implicit Adams method of the same order (right).}\label{fig:hires-exper}
\end{figure}


\paragraph{HIRES} This is a classical mildly stiff test system of dimension 8 describing a chemical reaction, see \cite[IV.10, (10.4)]{HW}. All equations except the 6-th and the 7-th are linear. The interval of integration is $[0,40]$. Figure \ref{fig:hires-exper} (left) shows the performance of 6-step stabilized methods of orders 1-6 and implicit method of order 6. We see that the results agree well with common sense: more accurate methods have shorter stability intervals. Then we compare methods of order 5 and display the results at Figure \ref{fig:hires-exper} (right), where we took $k$ from 9 to 15 in order to get larger stability intervals than the implicit method have. There is a clear evidence that methods with larger $k$ have larger error constants. We don't show the resuls of the damped first-order method, since the difference compared to the simple non-damped methods is negligible for this test problem.

\paragraph{Burgers' equation} The second problem is taken from \cite{Abdulle}. Consider a method of lines discretization of the one-dimensional nonlinear boundary value problem 
\begin{equation}\label{eq:burgers}
\renewcommand{\arraystretch}{1.5}
\begin{array}{lll}
	u_t+\left(\dfrac{u^2}{2}\right)_x=\mu u_{xx},\quad x\in[0,1],\quad t\in[0,2.5],\\
	u(x, 0)=1.5 x(1-x)^2,\\
	u(0,t)=u(1,t)=0.
\end{array}
\end{equation}
The spatial derivatives are approximated by standard central finite differences, the discretization step is $\Delta x=1/501$, so the dimension of the resulting ODE is 500. Jacobi matrix of this problem is not symmetric and complex eigenvalues occur for sufficiently small values of $\mu$. We took $\mu=0.005$ for which this is not the case at the starting point, but apparently non-real eigenvalues do emerge during integration.

The first experiment is similar to the one from the previous problem. The results are presented on the left side of Figure \ref{fig:burgers-exper}: we compare six-step methods of different orders. We see that the first-order method with damping allows for taking longer time steps than the non-damped one. This indicates that the solution generates non-real eigenvalues of Jacobi matrix. Hence it is unlikely to benefit from using stabilized methods with large $k$ and $p$, for which we don't have damping yet. Indeed, our experiments showed that these methods cannot take larger steps than implicit Adams methods of the same order, even if their stability interval is longer. That's why on the right side of Figure \ref{fig:burgers-exper} we compare only stabilized explicit methods of order one with the implicit Euler method. This experiment shows that damping is crucial for the general performance of a stabilized method. Another obvious conclusion is that the explicit methods are less accurate than the implicit one.

\begin{figure}
	\begin{center}
		\includegraphics[width=6.7cm]{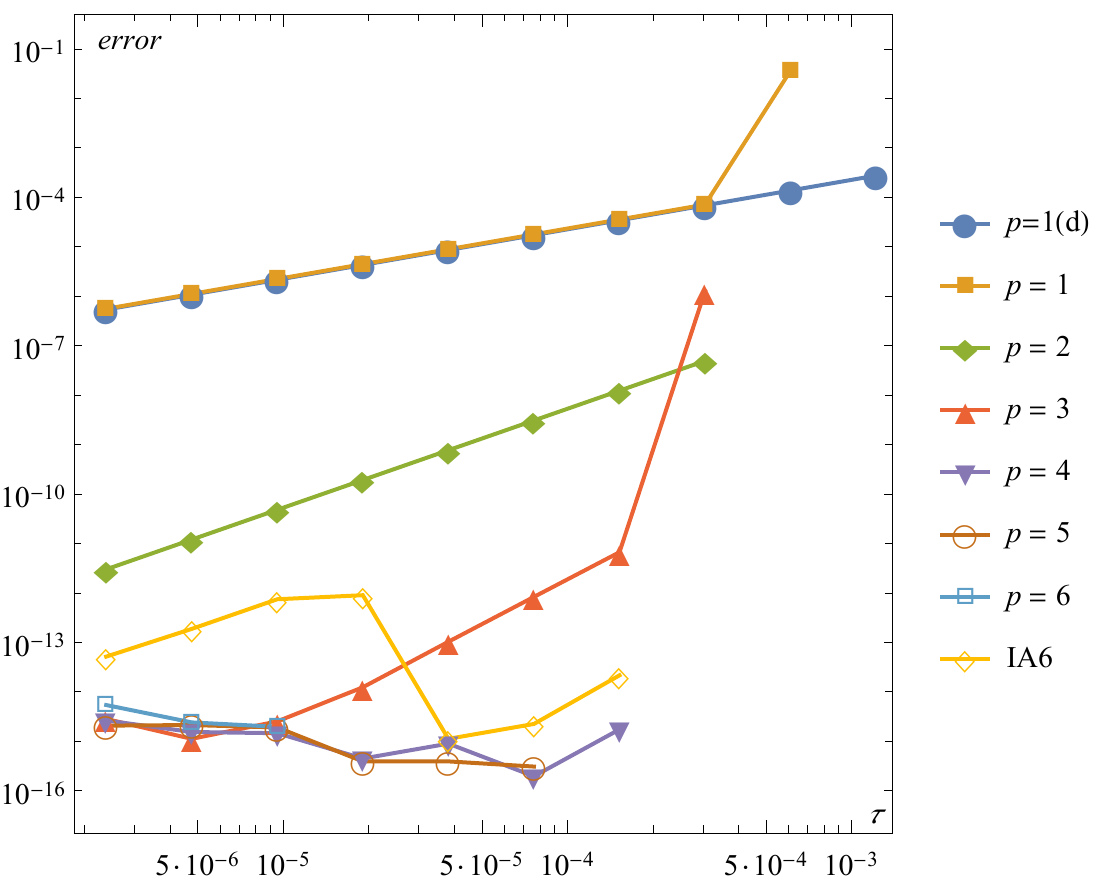}
		\includegraphics[width=6.7	cm]{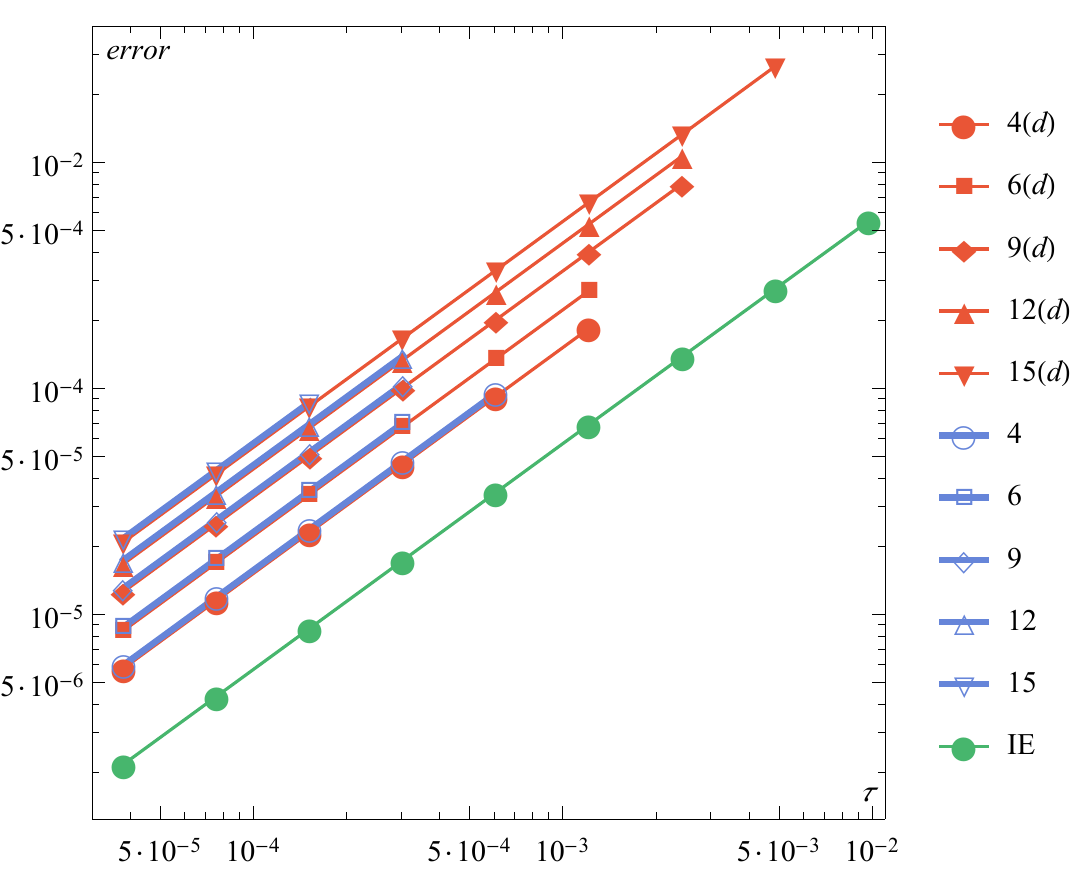}
	\end{center}
	\caption{Numerical experiment with Burgers equation \eqref{eq:burgers}.  Six-step stabilized methods and implicit Adams method of order~6 (left). First order stabilized methods with and without damping for $\varepsilon=0.25$, $k=4, 6, 9, 12, 15$, and the implicit Euler method (right).}\label{fig:burgers-exper}
\end{figure}

\section{Conclusion}\label{s:conclusion}

In this work we presented explicit multistep methods of Adams type, which possess extended stability intervals. A simple formulae for the first order methods and their error constant are derived. We also applied damping to the first order methods, derived a general scheme for construction of stabilized $k$-step methods of any order $p<k$, and calculated coefficients for such methods numerically.
It was shown that the error constant of the stabilized grows as the number of steps increases, but this growth is quite slow. 
Our numerical experiments asserted the theoretical properties of accuracy and stability of the constructed methods and exhibited the importance of damping transformation for the methods.

In our opinion, by now the stabilized Adams-type methods are mostly of theoretical interest. But it cannot be ruled out that they can be useful in practice and become a base for a competetive solver for mildly stiff problems. From a practical perspective the methods are attractive due to their low cost (just one evaluation of $f$ per step for any $k$ and $p$) and simplicity of implementation.   The weak point is that long stability intervals require large number of steps, which may entail overhead related to initialization of starting values, step size control and so on.

\paragraph{Acknowledgement} The work is supported by Belarusian government program of scientific research `Convergence-2020'.

\bibliographystyle{elsarticle-num-names} 
\bibliography{stabadams}

\begin{thebibliography}{9}
\expandafter\ifx\csname natexlab\endcsname\relax\def\natexlab#1{#1}\fi
\providecommand{\url}[1]{\texttt{#1}}
\providecommand{\href}[2]{#2}
\providecommand{\path}[1]{#1}
\providecommand{\DOIprefix}{doi:}
\providecommand{\ArXivprefix}{arXiv:}
\providecommand{\URLprefix}{URL: }
\providecommand{\Pubmedprefix}{pmid:}
\providecommand{\doi}[1]{\href{http://dx.doi.org/#1}{\path{#1}}}
\providecommand{\Pubmed}[1]{\href{pmid:#1}{\path{#1}}}
\providecommand{\bibinfo}[2]{#2}
\ifx\xfnm\relax \def\xfnm[#1]{\unskip,\space#1}\fi
\bibitem[{Hairer et~al.(1993)Hairer, N{\o}rsett, and Wanner}]{HW}
\bibinfo{author}{E.~Hairer}, \bibinfo{author}{S.~N{\o}rsett},
  \bibinfo{author}{G.~Wanner}, \bibinfo{title}{Solving Ordinary Differential
  Equations II: Stiff and Differential-Algebraic Problems}, Solving Ordinary
  Differential Equations II: Stiff and Differential-algebraic Problems,
  \bibinfo{publisher}{Springer}, \bibinfo{year}{1993}. \URLprefix
  \url{https://books.google.by/books?id=m7c8nNLPwaIC}.
\bibitem[{Lebedev(1994)}]{lebedev}
\bibinfo{author}{V.~Lebedev},
\newblock \bibinfo{title}{How to solve stiff systems of differential equations
  by explicit methods},
\newblock in: \bibinfo{booktitle}{Numerical Methods and Applications (1994)},
  \bibinfo{publisher}{CRC Press}, \bibinfo{year}{1994}, pp.
  \bibinfo{pages}{45--80}.
\bibitem[{Sommeijer et~al.(1998)Sommeijer, Shampine, and Verwer}]{sommeijer}
\bibinfo{author}{B.~Sommeijer}, \bibinfo{author}{L.~Shampine},
  \bibinfo{author}{J.~Verwer},
\newblock \bibinfo{title}{Rkc: An explicit solver for parabolic pdes},
\newblock \bibinfo{journal}{Journal of Computational and Applied Mathematics}
  \bibinfo{volume}{88} (\bibinfo{year}{1998}) \bibinfo{pages}{315 -- 326}.
  \URLprefix
  \url{http://www.sciencedirect.com/science/article/pii/S0377042797002197}.
  \DOIprefix\doi{https://doi.org/10.1016/S0377-0427(97)00219-7}.
\bibitem[{Abdulle(2002)}]{Abdulle}
\bibinfo{author}{A.~Abdulle},
\newblock \bibinfo{title}{Fourth order chebyshev methods with recurrence
  relation},
\newblock \bibinfo{journal}{SIAM Journal on Scientific Computing}
  \bibinfo{volume}{23} (\bibinfo{year}{2002}) \bibinfo{pages}{2041--2054}.
  \URLprefix \url{https://doi.org/10.1137/S1064827500379549}.
  \DOIprefix\doi{10.1137/S1064827500379549}.
  \href{http://arxiv.org/abs/https://doi.org/10.1137/S1064827500379549}{{\tt
  arXiv:https://doi.org/10.1137/S1064827500379549}}.
\bibitem[{Jeltsch and Nevanlinna(1981)}]{Jeltsch1981}
\bibinfo{author}{R.~Jeltsch}, \bibinfo{author}{O.~Nevanlinna},
\newblock \bibinfo{title}{Stability of explicit time discretizations for
  solving initial value problems},
\newblock \bibinfo{journal}{Numerische Mathematik} \bibinfo{volume}{37}
  (\bibinfo{year}{1981}) \bibinfo{pages}{61--91}. \URLprefix
  \url{https://doi.org/10.1007/BF01396187}. \DOIprefix\doi{10.1007/BF01396187}.
\bibitem[{Jeltsch and Nevanlinna(1982)}]{Jeltsch1982}
\bibinfo{author}{R.~Jeltsch}, \bibinfo{author}{O.~Nevanlinna},
\newblock \bibinfo{title}{Stability and accuracy of time discretizations for
  initial value problems},
\newblock \bibinfo{journal}{Numerische Mathematik} \bibinfo{volume}{40}
  (\bibinfo{year}{1982}) \bibinfo{pages}{245--296}. \URLprefix
  \url{https://doi.org/10.1007/BF01400542}. \DOIprefix\doi{10.1007/BF01400542}.
\bibitem[{Daubechies(1992)}]{Daubechies}
\bibinfo{author}{I.~Daubechies}, \bibinfo{title}{Ten lectures on wavelets},
  CBMS-NSF regional conference series in applied mathematics 61,
  \bibinfo{edition}{1} ed., \bibinfo{publisher}{Society for Industrial and
  Applied Mathematics}, \bibinfo{year}{1992}.
\bibitem[{Hairer et~al.(1993)Hairer, Norsett, and Wanner}]{HNW}
\bibinfo{author}{E.~Hairer}, \bibinfo{author}{S.~P. Norsett},
  \bibinfo{author}{G.~Wanner}, \bibinfo{title}{Solving Ordinary Differential
  Equations I. Nonstiff Problems}, \bibinfo{edition}{2nd rev. ed. 1993. corr.
  3rd printing} ed., \bibinfo{publisher}{Springer}, \bibinfo{address}{Berlin},
  \bibinfo{year}{1993}. \URLprefix
  \url{https://archive-ouverte.unige.ch/unige:12346}, \bibinfo{note}{iD:
  unige:12346}.
\bibitem[{Xu and Zhao(2010)}]{XuZhao}
\bibinfo{author}{Y.~Xu}, \bibinfo{author}{J.~J. Zhao},
\newblock \bibinfo{title}{Estimation of longest stability interval for a kind
  of explicit linear multistep methods},
\newblock \bibinfo{journal}{Discrete Dynamics in Nature and Society}
  \bibinfo{volume}{2010} (\bibinfo{year}{2010}) \bibinfo{pages}{1--18}.
  \URLprefix \url{https://EconPapers.repec.org/RePEc:hin:jnddns:912691}.

\end{thebibliography}

\begin{figure*}
	\begin{center}
		\includegraphics[width=\textwidth]{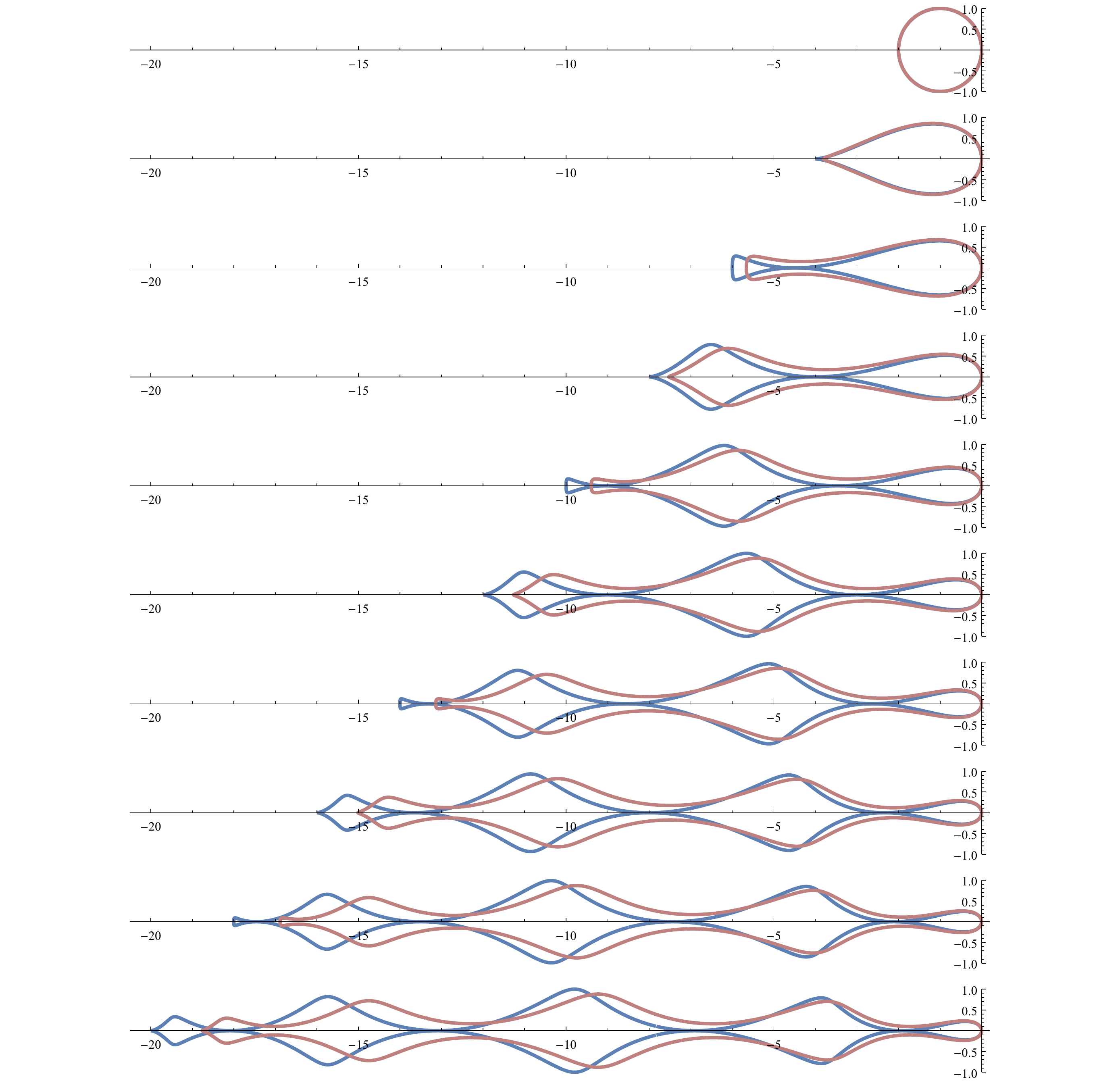}
	\end{center}
	\caption{Stability regions of the optimized first order methods and their damped versions ($\varepsilon=0.25$), for $k=1,\ldots,10.$}
	\label{fig:order1}
\end{figure*}

\begin{figure*}
	\includegraphics[width=\textwidth]{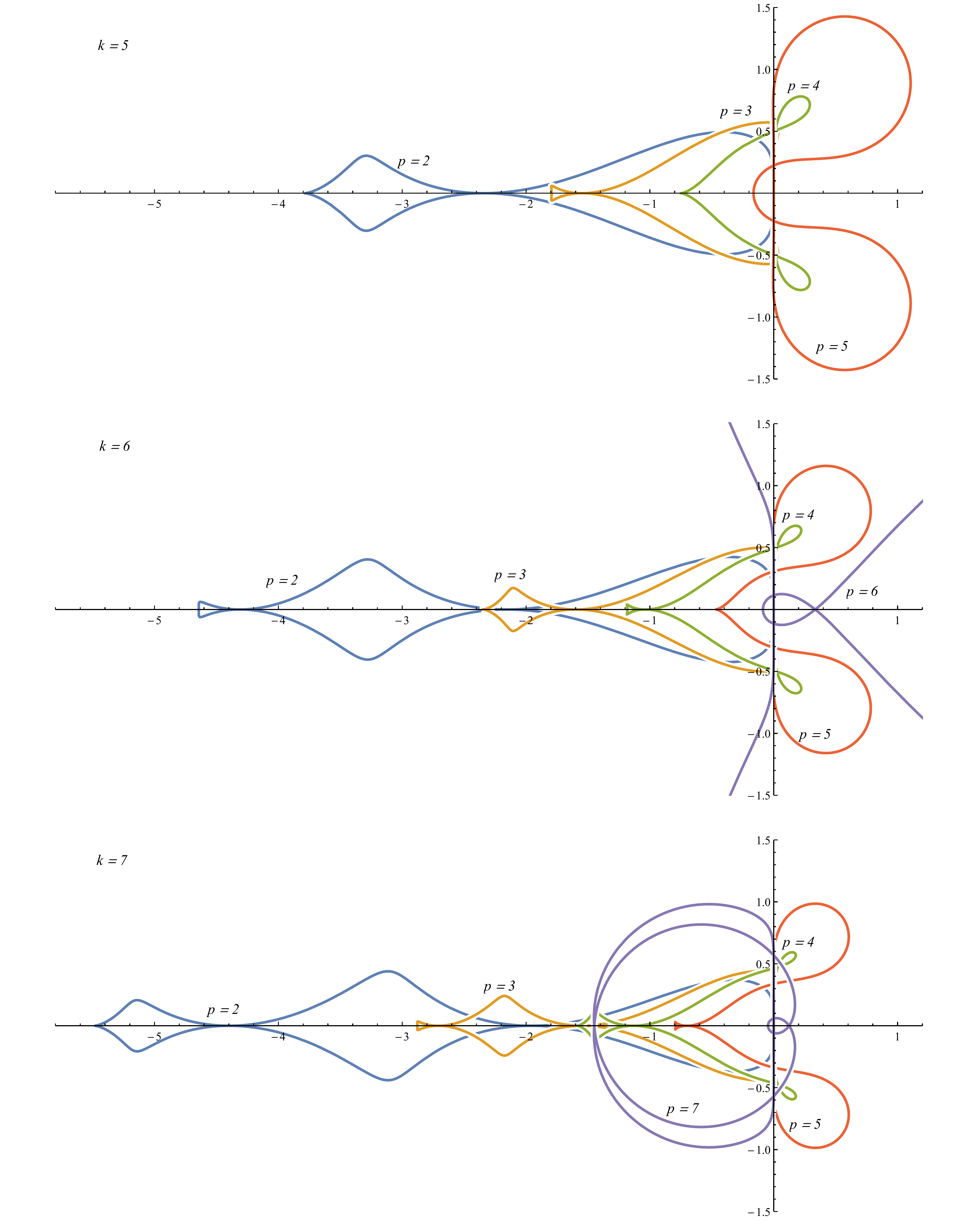}
	\caption{Stability regions of the multistep optimized methods for $k=5, 6, 7$.}
	\label{fig:k567} 
\end{figure*}

\begin{table*}
\scriptsize
\[
\begin{array}{lllll}
k
&
\multicolumn{1}{c}{\text{Order 2}} & 
\multicolumn{1}{c}{\text{Order 3}} &
\multicolumn{1}{c}{\text{Order 4}} &
\multicolumn{1}{c}{\text{Order 5}} \\
\hline\hline
 3 & \ell= 2 & \ell=0.545454545454545455 &    & \text{} \\
  & -0.25 & 0.41666666666666666667 &    &  \\
  & 0 & -1.3333333333333333333 &    &  \\  
  & 1.25 & 1.9166666666666666667  &    &  \\ 
 \hline
 4 & \ell=2.914213562373095 & \ell=1.2 &   \ell=0.3 & \text{} \\
  & -0.14644660940673046069 & 0.25 & -0.37500000000000000000 & \text{}   \\
 \text{} & -0.18198051533945963691 & -0.33333333333333333333 & 1.5416666666666666667 &   \text{} \\
 \text{} & 0.30330085889911065590 & -0.58333333333333333333 & -2.4583333333333333333 & \text{}   \\
 \text{} & 1.0251262658470794417 & 1.6666666666666666667 &  2.2916666666666666667 & \text{} \\
 \hline
 5 & \ell=3.788854381999832 & \ell=1.793779334348686 &   \ell=0.75 & \ell=0.1633393829401088 \\
 \text{} & -0.095491502812526287949 & 0.16437694101246125619 &   -0.25 & 0.34861111111111111111  \\
 \text{} & -0.17705098312484227231 & -0.0097910917750136439271 &   0.625 & -1.7694444444444444444  \\
 \text{} & 0 & -0.54022170408305993867 & 0.041666666666666666667 & 3.6333333333333333333  \\
 \text{} & 0.41311896062463196872 & -0.047691080558684215630 &   -1.4583333333333333333 & -3.8527777777777777778 \\
 \text{} & 0.85942352531273659154 & 1.4333269354042965420 &   2.0416666666666666667 &  2.6402777777777777778 \\
   \hline
 6 & \ell=4.642734410091836 & \ell=2.347826086956522 &   \ell=1.181897711989360 & \ell=0.469157254561251 \\
 \text{} & -0.066987298107786995665 & 0.11574074074074731606 &   -0.17622805914576966884 & 0.24942129629629629629 \\
 \text{} & -0.14711431702997807715 & 0.087962962962956387640 &   0.21777962430380174276 & -0.89849537037037037036 \\
 \text{} & -0.089745962155603046598 & -0.28703703703705018768 &   0.51616209424248971734 & 0.72476851851851851851 \\
 \text{} & 0.12564434701786943107 & -0.40740740740739425676 &   -0.67621677042591625354 & 1.1391203703703703704 \\
 \text{} & 0.44134295108991756459 & 0.24537037037037694569 &   -0.68603094336199527180 & -2.6056712962962962963 \\
 \text{} & 0.73686027918558112375 & 1.2453703703703637950 &   1.8045340543873897341 & 2.3908564814814814815 \\
 \hline
 7 & \ell=5.484476959454063 & \ell=2.877558710633067 &
   \ell=1.586803103995642 & \ell=0.792362028995767 \\
 \text{} & -0.049515566048790436882 & 0.085721156820309456282 &
   -0.13027657069924974882 & 0.18480570522895041008 \\
 \text{} & -0.11912520277278577227 & 0.11154612811463327941 &
   0.040823321662060514133 & -0.43546201769087620565 \\
 \text{} & -0.11018250002552420585 & -0.11721033134808636645 &
   0.45157410201399110594 & -0.24616437886876401181 \\
 \text{} & 0 & -0.35463665779124584907 & 0.016001789308425411316 &
   1.2681635878048099022 \\
 \text{} & 0.19832850004594357054 & -0.21744000532205222576 &
   -0.79486796947441511195 & -0.32830322506067306370 \\
 \text{} & 0.43679241016688116501 & 0.39557372791516432979 &
   -0.18702302114721451156 & -1.5947509407373489642 \\
 \text{} & 0.64370235863427567947 & 1.0964459816112773758 &
   1.6037683483364023409 & 2.1517112693239019330 \\
 \hline
 8 & \ell=6.318535592272045 & \ell=3.391689975797208 &
   \ell=1.970916561391601 & \ell=1.105498503602666 \\
 \text{} & -0.038060233744366798686 & 0.065966021983597280828 &
   -0.10001878254782277331 & 0.14160831078216433500 \\
 \text{} & -0.096797724520983369102 & 0.11032441087323208003 &
   -0.035748890463804216949 & -0.19477889703735130008 \\
 \text{} & -0.10779695287351088696 & -0.022713554363876414313 &
   0.30658371766113087497 & -0.45247252238839228671 \\
 \text{} & -0.052994558379770972895 & -0.23691021182637878968 &
   0.27749085554924180902 & 0.57636630123759032103 \\
 \text{} & 0.068135860774038963863 & -0.30634225579338833174 &
   -0.33516540035839458087 & 0.78354708230483585981 \\
 \text{} & 0.23715329632173532873 & -0.055862376110155554778 &
   -0.67199165170356770075 & -0.91953823465464150558 \\
 \text{} & 0.41945680625751858277 & 0.46825151960079988906 &
   0.12122231459728224806 & -0.87725216386466696327 \\
 \text{} & 0.57090350616533915229 & 0.97728644563616984059 &
   1.4376278372659343398 & 1.9425201236204615397 \\
 \hline
 9 & \ell=7.147430550561413 & \ell=3.895290219607647 &
   \ell=2.339983407348191 & \ell=1.405151117615213 \\
 \text{} & -0.030153689607037932268 & 0.052301051895272605013 &
   -0.079129092227346338565 & 0.11167958745225367479 \\
 \text{} & -0.079550128858107345641 & 0.10126696210118874790 &
   -0.067460438055823679907 & -0.068703909200215014827 \\
 \text{} & -0.098407115533249091604 & 0.026642016446313140531 &
   0.18522989963169925608 & -0.41559134883278779976 \\
 \text{} & -0.073305865502781992742 & -0.13793192915668298604 &
   0.31675641768693750027 & 0.075957984853647800951 \\
 \text{} & 0 & -0.26695490513260400200 & 0.0076996887855987555993 &
   0.78975711968445738645 \\
 \text{} & 0.11519493150433742625 & -0.21907659037234928746 &
   -0.48561642796053139031 & 0.16879857406276817077 \\
 \text{} & 0.25585850038645603291 & 0.064279256258874647289 &
   -0.48641107197220078201 & -1.0382277451316307602 \\
 \text{} & 0.39775064429061670700 & 0.49902127636474858744 &
   0.30896699066825414262 & -0.38771987715090903834 \\
 \text{} & 0.51261272331978661124 & 0.88045286159523854733 &
   1.2999640334434125362 & 1.7640496142624155802 \\
 \hline
 10 & \ell=7.972691637812280 & \ell=4.391469108714782 &
   \ell=2.698087099023256 & \ell=1.692885048664239 \\
 \text{} & -0.024471741852422821505 & 0.042467110956300544552 &
   -0.064133502960306610717 & 0.090219510737302839601 \\
 \text{} & -0.066228831765768206903 & 0.090440497652067647206 &
   -0.078573353260495406661 & -0.0021584562050617957037 \\
 \text{} & -0.087599164129385382526 & 0.051030056647860180918 &
   0.099782736471490155539 & -0.32195487552605745395 \\
 \text{} & -0.078738975641538713579 & -0.068250050077061163328 &
   0.27409149956975402355 & -0.17148478569282268595 \\
 \text{} & -0.034883488233566344682 & -0.19902094851262917934 &
   0.17521906381042658379 & 0.47486789482155684885 \\
 \text{} & 0.042635374507685291073 & -0.24395517042504782618 &
   -0.20265793719790100791 & 0.59839764726184595395 \\
 \text{} & 0.14622952619142684103 & -0.12913104538091815896 &
   -0.50346262595639964788 & -0.27671853444446566397 \\
 \text{} & 0.26279749238816316420 & 0.14896994335213981908 &
   -0.30713843196368842739 & -0.94638400314820567730 \\
 \text{} & 0.37529671333936471557 & 0.50694059780591167508 &
   0.42196137154381443077 & -0.057121557681252610888 \\
 \text{} & 0.46496309519604145733 & 0.80050900798137646097 &
   1.1849111799433059069 & 1.6123371598771602453 \\
\end{array}
\]
\normalsize
\caption{Coefficients and stability interval lengths of the optimized methods of order 2-5}
\label{tab:o2-5}
\end{table*}

\begin{table*}
\scriptsize
\[
\begin{array}{lllll}
k
&
\multicolumn{1}{c}{\text{Order 6}} & 
\multicolumn{1}{c}{\text{Order 7}} &
\multicolumn{1}{c}{\text{Order 8}} &
\multicolumn{1}{c}{\text{Order 9}} 
\\\hline\hline
6 & \ell=0.08771929824561404 & \text{} & \text{} & \text{} \\
 \text{} & -0.32986111111111111111 & \text{} & \text{} & \text{} \\
 \text{} & 1.9979166666666666667 & \text{} & \text{} & \text{} \\
 \text{} & -5.0680555555555555556 & \text{} & \text{} & \text{} \\
 \text{} & 6.9319444444444444444 & \text{} & \text{} & \text{} \\
 \text{} & -5.5020833333333333333 & \text{} & \text{} & \text{} \\
 \text{} & 2.9701388888888888889 & \text{} & \text{} & \text{} \\
\hline
7 & \text{NOT CONVERGED} & \ell=0.04651391725937046 &
   \text{} & \text{} \\
 \text{} &  & 0.31559193121693121693 & \text{} & \text{}
   \\
 \text{} &  & -2.2234126984126984127 & \text{} & \text{}
   \\
 \text{} &  & 6.7317956349206349206 & \text{} & \text{}
   \\
 \text{} &  & -11.379894179894179894 & \text{} &
   \text{} \\
 \text{} &  & 11.665823412698412698 & \text{} & \text{} \\
 \text{} &  & -7.3956349206349206349 & \text{} & \text{}
   \\
 \text{} &  & 3.2857308201058201058 & \text{} & \text{} \\
\hline
8 & \ell=0.5290722934773335 & \text{NOT CONVERGED} &   \ell=0.02440851327616489 & \text{} \\
 \text{} & -0.19113689616832294585 &  &   -0.30422453703703703704 & \text{} \\
 \text{} & 0.65850013289950086628 &  &   2.4451636904761904762 & \text{} \\
 \text{} & -0.26698708897333444593 &  &   -8.6121279761904761905 & \text{} \\
 \text{} & -1.5041640716234265487 &  &   17.379654431216931217 & \text{} \\
 \text{} & 1.8313158841283364334 &  &   -22.027752976190476190 & \text{} \\
 \text{} & 0.75394715979782998677 &  &   18.054538690476190476 & \text{} \\
 \text{} & -2.7632927648390354259 &  &   -9.5252066798941798942 & \text{} \\
 \text{} & 2.4818176447784520799 &  &   3.5899553571428571429 & \text{} \\
\hline
9 & \ell=0.7745044113664562 & \text{NOT CONVERGED} &   \text{NOT CONVERGED} & \ell=0.01270447596389330 \\
 \text{} & -0.15072405770953055168 &  &    & 0.29486800044091710758 \\
 \text{} & 0.36616417962483152368 &  &    & -2.6631685405643738977 \\
 \text{} & 0.26486240742135927331 &  &    & 10.701467702821869489 \\
 \text{} & -1.1679360960270178460 &  &    & -25.124736000881834215 \\
 \text{} & -0.049706767276153478305 &  &    & 38.020414462081128748 \\
 \text{} & 1.8307408258122144817 &  &    & -38.540361000881834215 \\
 \text{} & -0.54006451441787310785 &  &    & 26.310842702821869489 \\
 \text{} & -1.8201171395869259716 &  &    & -11.884150683421516755 \\
 \text{} & 2.2667811621590956802 &  &    & 3.8848233575837742504 \\
\hline
10 & \ell=1.015322150308401 & \text{NOT CONVERGED} &   \text{NOT CONVERGED} & \text{NOT CONVERGED} \\
 \text{} & -0.12149925981588955161 &  &   &  \\
 \text{} & 0.19502001210515154522 &  &   & \\
 \text{} & 0.40323654967363550399 &   &   &  \\
 \text{} & -0.60200414081780015659 &   &   &  \\
 \text{} & -0.79801775043705878458 &   &   &  \\
 \text{} & 0.91298862642764008111 &   &   &  \\
 \text{} & 1.1648437230850238167 &   &   &  \\
 \text{} & -1.1001111732352200672 &   &   &  \\
 \text{} & -1.1334723376167517028 &   &   &  \\
 \text{} & 2.0790157506312693158 &   &   &  
\end{array}
\]
\normalsize
\caption{Coefficients and stability interval lengths of the optimized methods of order 6}
\label{tab:o6-9}
\end{table*}

\newpage
\appendix
\section{\emph{Mathematica} code for computing the stabilized method's parameters}\label{app:code}
\footnotesize
\begin{verbatim}
ClearAll[a, b, beta, oc, mu]
k = 5; o = 3;
mu[betas_List] := With[{k = Length@betas}
   , Evaluate[(#^k - #^(k - 1))/(#^Range[0, k - 1]).betas] &
   ];
param = {a[-1] -> 0, a[k] -> 0
   , a[k - 1] -> Sum[b[j]^2, {j, 0, k - 1}]
   , a[j_] :> Sum[b[l] b[l + k - j - 1], {l, 0, j}]
   , beta[k - 1] -> a[k - 1] + a[k - 2]
   , beta[j_] :> a[j - 1] + a[j]
   };
bs = b /@ Range[0, k - 1];
betas = beta /@ Range[0, k - 1];
oc[1] = Total@betas - 1;
oc[p_] := Simplify[betas.Range[1 - k, 0]^(p - 1)] - 1/p;
cons = Thread[(oc /@ Range[o] //. param) == 0];
sol = NMinimize[Prepend[cons, bs.bs], bs
   , Method -> Automatic
   , WorkingPrecision -> 50
   , AccuracyGoal -> 25
   , PrecisionGoal -> 25
   , MaxIterations -> 1000
   ];
betaopt = (betas //. param) /. sol[[2]];
rescond = (oc /@ Range[o]) /. Thread[betas -> betaopt];
<|"k" -> k, "order" -> o, "betas" -> betaopt, "orderres" -> rescond, 
 "len" -> -mu[betaopt][-1]|>
\end{verbatim}





\end{document}